\documentclass[11pt]{article}

\usepackage{amsfonts}
\usepackage{amssymb,amsmath,amsthm}
\usepackage{latexsym}
\usepackage{fullpage}
\usepackage{hyperref}

\newtheorem{theorem}{Theorem}[section]

\newtheorem{lemma}[theorem]{Lemma}
\newtheorem{cor}[theorem]{Corollary}

\newtheorem{fact}[theorem]{Fact}

\newtheorem{claim}[theorem]{Claim}

\theoremstyle{definition}

\newcounter{tenumerate}

\renewcommand{\epsilon}{\varepsilon}

\newcommand{\Lip}{\mathrm{Lip}}

\newcommand{\dist}{\mathsf{dist}}
\renewcommand{\dim}{\mathsf{dim}}

\newcommand{\grad}{\nabla}

\newcommand{\remove}[1]{}

\renewcommand{\leq}{\leqslant}
\renewcommand{\geq}{\geqslant}
\newcommand{\diam}{\mathsf{diam}}
\newcommand{\f}{\varphi}

\newcommand{\laplace}{\Delta}

\begin{document}

\title{{\bf Eigenvalue multiplicity and volume growth}}

\author{James R. Lee\thanks{Research supported by NSF CCF-0644037.} \\ {\small University of Washington}
\and Yury Makarychev \\ {\small Microsoft Research}}
\date{}

\maketitle

\begin{abstract}
Let $G$ be a finite group with symmetric generating set $S$, and let
$c = \max_{R > 0} \frac{|B(2R)|}{|B(R)|}$ be the
doubling constant of the corresponding Cayley graph, where $B(R)$ denotes an $R$-ball
in the word-metric
with respect to $S$.
We show
that the multiplicity of the $k$th eigenvalue of the Laplacian on the Cayley
graph of $G$
is bounded by a function of only $c$ and $k$.  More specifically,
the multiplicity is at most $\exp\left(O(\log c)(\log c + \log k)\right)$.

Similarly, if $X$ is a compact, $n$-dimensional Riemannian manifold
with non-negative Ricci curvature, then the multiplicity
of the $k$th eigenvalue of the Laplace-Beltrami operator on $X$
is at most $\exp\left(O(n) (n + \log k))\right)$.

The first result (for $k=2$) yields the following group-theoretic application.
There exists a normal subgroup $N$ of $G$, with $[G : N] \leq \alpha(c)$, and
such that $N$ admits a homomorphism onto $\mathbb Z_M$, where
$M \geq |G|^{\delta(c)}$ and
\begin{eqnarray*}
\alpha(c) &\leq& O(h)^{h^2} \\
\delta(c) &\geq& \frac{1}{O(h \log c)},
\end{eqnarray*}
where $h \leq \exp((\log c)^2)$.
This is an effective, finitary analog of
a theorem of Gromov which states that every infinite group of polynomial growth
has a subgroup of finite index which admits a homomorphism onto $\mathbb Z$.

This addresses a question of Trevisan, and is proved by scaling down
Kleiner's proof of Gromov's theorem.  In particular,
we replace the space of harmonic functions of fixed polynomial growth
by the second eigenspace of the Laplacian on the Cayley graph of $G$.
\remove{
Finally, we use this to show that there exists a constant $K = K(c)$ such that $G$
contains a $K$-step normal nilpotent subgroup $H$ with $[G : H] \leq K$,
yielding a finitary version of Gromov's theorem on infinite groups
of polynomial growth.}
\end{abstract}

\section{Introduction}

Let $G$ be a finitely generated group with finite, symmetric
generating set $S$. The Cayley graph $\mathrm{Cay}(G;S)$ is an
undirected $|S|$-regular graph with vertex set $G$ and an edge
$\{u,v\}$ whenever $u = vs$ for some $s \in S$.  We equip $G$ with
the natural word metric, which is also the shortest-path
metric on $\mathrm{Cay}(G;S)$.  Letting $B(R)$ be the closed ball of
radius $R$ about $e \in G$, one says that $G$ has {\em polynomial
growth} if there exists a number $m \in \mathbb N$ such that
$$
\lim_{R \to \infty} \frac{|B(R)|}{R^m} < \infty.
$$
It is easy to see that this property is independent of the choice
of finite generating set $S$.

In a classical paper \cite{Gromov81}, Gromov proved that a group has polynomial growth
if and only if it contains a nilpotent subgroup of finite index.  The
sufficiency part was proved earlier by Wolf \cite{Wolf68}.  It is natural to ask
about similar phenomenon holds in {\em finite groups}.  Of course, every
finite group has polynomial growth trivially, so even formulating a
similar question is not straightforward.
As Gromov points out \cite{Gromov81}, by a compactness argument,
one only needs $|B(R)| \leq C R^m$ to hold for $R \leq R_0$, for some
$R_0 = R_0(C,|S|,m)$.  Thus one can formulate a version of Gromov's theorem
for finite groups.  However, there are no effective estimates known
for $R_0$.
Furthermore, Gromov's proof
relies on a limiting procedure which is again trivial for finite groups.

Recently, Kleiner \cite{Kleiner07} gave a new proof of Gromov's theorem
that avoids the limiting procedure, and in particular avoids
the use of the Yamabe-Montgomery-Zippin 
structure theory \cite{MZ74} to classify
the limit objects.  The main step of Kleiner's proof lies in showing
that the space of harmonic functions of fixed polynomial growth
is finite-dimensional on an infinite group $G$ of polynomial growth.  Such a result follows
from the work of Colding and Minicozzi \cite{CM97}, but their proof
uses Gromov's theorem, whereas Kleiner is able to obtain the
result essentially from scratch, based on a new
scale-dependent Poincar\'e inequality for bounded-degree graphs.

Again, the connection with finite groups is lacking:  Every harmonic
function on a finite graph is constant.  In the present work, we show
that one can obtain some {\em effective}
partial analogs of Gromov's theorem for finite
groups by following Kleiner's general outline, but replacing the
space of harmonic functions of fixed polynomial growth with the
second eigenspace of the discrete Laplacian on $\mathrm{Cay}(G;S)$.

We recall the following two theorems of Gromov, which capture the
essential move
from a geometric condition (polynomial volume growth of balls) on
an infinite group $G$, to a conclusion about its algebraic structure.

\begin{theorem}[Gromov \cite{Gromov81}]
\label{thm:gromov}
If $G$ is an infinite group of polynomial growth, the following holds.
\begin{enumerate} \item
$G$
admits a finite-dimensional linear representation $\rho : G \to GL_n(\mathbb C)$
such that $\rho(G)$ is infinite.
\item $G$ contains a normal subgroup $N$, with $[G : N] = O(1)$,
and such that $N$ admits a homomorphism onto $\mathbb Z$.
\end{enumerate}
\end{theorem}

In fact, by the simplifications of Tits (in Appendix A.2 of \cite{Gromov81}), Gromov's
theorem follows fairly easily using an induction on (2).  After seeing Kleiner's
proof, Luca Trevisan asked whether there is a quantitative analog of part (1) of Theorem \ref{thm:gromov} for finite groups.  We prove the following.

\begin{theorem}\label{thm:mainst}
Let $G$ be a finite group.  For any symmetric generating set $S$, define
$$c = c_{G;S} = \max_{R \geq 0} \frac{|B(2R)|}{|B(R)|},$$
where $B(\cdot)$ is a closed ball in $\mathrm{Cay}(G;S)$.
Then the following holds.
\begin{enumerate}
\item There is a linear representation $\rho : G \to GL(\mathbb R^k)$, where
$k \leq \exp\left(O(\log c)^2\right)$, and
$|\rho(G)| \geq c^{-O(1)} |G|^{1/\log_2(c)}$.
\item There is a normal subgroup $N \leq G$, with $[G : N] = O(k)^{k^2}$,
and $N$ admits a homomorphism onto $\mathbb Z_M$, where
$M \geq c^{-O(1)} |G|^{1/(k \log_2 c)}$.
\end{enumerate}
\end{theorem}

Observe that we have assumed a bound on the ratios $|B(2R)|/|B(R)|$, which is stronger
than an assumption of the form
\begin{equation}\label{eq:standgrowth}
|B(R)| \leq C R^m \qquad \textrm{for some $C, m > 0$.}
\end{equation}
The latter type of condition seems far more unwieldy in the setting of finite groups.
By making such an assumption, we completely bypass a ``scale selection'' argument, and
the delicacy required by Kleiner's approach (which has to perform many steps of
the proof using only the geometry at a single scale).  All of our arguments
  can be carried out at a single scale (see, e.g. the Reverse Poincar\'e Inequality
  for graphs in Section \ref{sec:reverse}), but it is not clear whether
there is an appropriate, effective scale selection procedure in the finite case, and we leave
the extension of Theorem \ref{thm:mainst} to a bounded growth condition like \eqref{eq:standgrowth} as an interesting open question.

\remove{
Finally, note that the ideal conclusion of such a study is a statement of the form:
There exists a normal subgroup $N \leq G$ such that $[G : N] = O(1)$, and $N$
is an $O(1)$-step nilpotent group, where the $O(1)$ notation hides a constant
that depends only on the growth data of $G$.  It is not clear that such a strong
property can hold under only the assumptions of Theorem \ref{thm:mainst}.
One might need to make assumptions on {\em every} generating set of $G$,
or even geometric assumptions on families of subgroups of $G$.  Defining
a simple condition on $G$ and its generators that can achieve the full
algebraic conclusion is an intriguing open problem.}

\subsection{Proof outline and eigenvalue multiplicity}

Our proof of Theorem \ref{thm:mainst} proceeds along the following lines.
Given an undirected $d$-regular graph $H = (V,E)$, one defines
the {\em discrete Laplacian on $H$} as the operator $\Delta : L^2(V) \to L^2(V)$
given by $\Delta(f)(x) = f(x) - \frac{1}{d} \sum_{y : \{x,y\} \in E} f(y)$.
The eigenvalues of $\Delta$ are non-negative and can be ordered
$\lambda_1 = 0 \leq \lambda_2 \leq \cdots \leq \lambda_n$, where $n = |V|$.
The {\em second eigenspace of $\Delta$} is the subspace $W_2 \subseteq L^2(V)$
given by $W_2 = \{ f \in L^2(V) : \Delta f = \lambda_2 f \}$.  Finally, the well-known (geometric)
{\em multiplicity of $\lambda_2$} is precisely $\dim(W_2)$.

In Section \ref{sec:eigen}, we use the approach of Colding and Minicozzi \cite{CM97} and Kleiner \cite{Kleiner07}
to argue that $\dim(W_2) = O(1)$,
whenever $c_H = \max_{x \in V, R \geq 0} \frac{|B(x,2R)|}{|B(x,R)|} = O(1)$,
and $H$ satisfies a certain Poincar\'e inequality.
At the heart of the proof lies the intuition
that functions in $W_2$ are the ``most harmonic-like'' functions on $H$ which are orthogonal
to the constant functions.  Carrying this out requires precise quantitative control
on the eigenvalues of $H$ in terms of $c_H$, which we obtain in Section \ref{sec:eigenbounds}.

\medskip

Now, consider $H = \mathrm{Cay}(G;S)$ for some finite group $G$, and the natural action of $G$ on $f \in L^2(G)$ given by $g f(x) = f(g^{-1} x)$.  It is easy to see that
this action commutes with the action of the Laplacian, hence $W_2$ is an invariant
subspace.  Since $\dim(W_2) = O(1)$, we will have achieved Theorem 1.2(1) as long
as the image of the action is large.  In Section \ref{sec:groupapp}, we show that
if the image of the action is small, then we can pass to a small quotient group,
and that $f$ pushes down to an eigenfunction on the quotient.  This allows us
to bound $\lambda_2$ on the quotient group in terms of $\lambda_2$ on $G$.
But $\lambda_2$ on a small, connected graph cannot be too close to zero by the
discrete Cheeger
inequality.  In this way, we arrive at a contradiction if the image of the action
is too small.  Theorem 1.2(2) is then a simple corollary of Theorem 1.2(1), using a theorem of Jordan on finite linear groups.

\medskip
\noindent
{\bf Higher eigenvalues and non-negatively curved manifolds.}
In fact, the techniques of Section \ref{sec:eigen} give
bounds on the multiplicity of higher eigenvalues of the Laplacian as well,
and the graph proof extends rather easily to bounding the eigenvalues of
the Laplace-Beltrami operator on Riemannian manifolds of non-negative
Ricci curvature.

Cheng \cite{Cheng76} proved that the multiplicity of the $k$th eigenvalue
of a compact Riemannian surface of genus $g$ grows like $O(g+k+1)^2$.
Besson later showed \cite{Besson80} that the multiplicity of the first non-zero eigenvalue
is only $O(g+1)$.  We refer to the book of Schoen and Yau \cite[Ch. 3]{SY94}
for further discussion of eigenvalue problems on manifolds.
In Section \ref{sec:eigen}, we prove a bound on the multiplicity of
the $k$th smallest non-zero eigenvalue of the Laplace-Beltrami operator
on compact Riemannian manifolds with non-negative Ricci curvature.
In particular, the multiplicity is bounded by a function depending only on $k$ and the dimension.
The main additional fact we require is an eigenvalue estimate
of Cheng \cite{Cheng75} in this setting.



\section{Preliminaries}

\subsection{Notation}

For $N \in \mathbb N$, we write $[N]$ for $\{1,2,\ldots,N\}$.

Given two expressions $E$ and $E'$ (possibly depending on a number of parameters), we write $E = O(E')$ to mean that $E \leq C E'$
for some constant $C > 0$ which is independent of the parameters. Similarly, $E = \Omega(E')$ implies that $E \geq C E'$ for some $C > 0$.
We also write $E \lesssim E'$ as a synonym for $E = O(E')$.  Finally, we write $E \approx E'$ to denote
the conjunction of $E \lesssim E'$ and $E \gtrsim E'$.

In a metric space $(X,d)$, for a point $x \in X$, we use $B(x,R) = \{ y \in X : d(x,y) \leq R \}$ to denote
the closed ball in $X$ about $x$.

\subsection{Laplacians, eigenvalues, and the Poincar\'e inequality}

Let $(X,\dist,\mu)$ be a metric-measure space.  Throughout the paper, we will be in one of the following two situations.
\begin{enumerate}
\item[({\sf G})] $X$ is a finite, connected, undirected $d$-regular graph, $\dist$ is the shortest-path metric, and $\mu$ is the counting measure.
In this case, we let $E(X)$ denote the edge set of $X$, and we write $y \sim x$ to denote $\{x,y\} \in E(X)$.

\remove{
Furthermore, we assume that $X$ satisfies a doubling condition on balls:  For some $c_X \geq 2$,
for every $x \in X$ and $R \geq 0$, we have
\begin{equation}\label{eq:doublingX}
\mu(B(x,2R)) \leq c_X\cdot \mu(B(x,R)).
\end{equation}
Observe that $d \leq c_X$ (taking $R=3/4$ above).}

\item[({\sf M})] $X$ is a compact $n$-dimensional Riemannian manifold without boundary, $\dist$ is the Riemannian distance, and $\mu$ is the Riemannian volume.
\remove{Furthermore, we assume that $X$ has non-negative Ricci curvature.  In particular, by standard volume comparison theorems (see, e.g. \cite{Karcher89}),
$X$ satisfies \eqref{eq:doublingX} with $c_X = 2^n$.}
\end{enumerate}

Since the proofs of Section \ref{sec:eigen} proceed virtually identically in both cases, we collect here
some common notation.
We define $\|f\|_2 = \left(\int f^2\,d\mu\right)^{1/2}$ for a function $f : X \to \mathbb R$,
and let $L^2(X)=L^2(X,\mu)$ be the Hilbert space of scalar functions for which $\|\cdot\|_2$ is bounded.
In the graph setting, we define the gradient by $[\grad f](x) = \frac{1}{\sqrt{2d}} (f(x)-f(y_1), \ldots, f(x)-f(y_d))$,
where $y_1, \ldots, y_d$ enumerate the neighbors of $x \in X$.  The actual order of enumeration
is unimportant as we will be primarily concerned with the expression $|\grad f(x)|^2 = \displaystyle \frac1{2d} \sum_{y : y \sim x} |f(x)-f(y)|^2$.

We define the Sobolev space $$L_1^2(X) = \left\{ f : \int f^2\,d\mu + \int |\grad f|^2\,d\mu < \infty \right\} \subseteq L^2(X).$$
Now we proceed to define the Laplacian $\Delta : L_1^2(X) \to L_1^2(X)$.

\begin{enumerate}
\item In the graph setting, $[\Delta f](x) = f(x) - \displaystyle \frac{1}{d} \sum_{y : y \sim x} f(y)$.
\item In the Riemannian setting, $\Delta$ is the Laplace-Beltrami operator.
\end{enumerate}

It is well-known that in both our settings, $\Delta$ is a self-adjoint operator on $L_1^2(X)$ with
discrete eigenvalues $0 = \lambda_1 < \lambda_2 \leq \lambda_3 \leq \cdots$.  In the graph
case, this sequence terminates with $\lambda_{|X|}$.  (Note that we have used the graph-theoretic
convention for numbering the eigenvalues; in the Riemannian setting, our $\lambda_1$ is usually
written as $\lambda_0$.)

We define the {\em $k$th eigenspace} by $$W_k = \{ \varphi \in L^2_1(X) : \Delta \varphi = \lambda_k \varphi \}$$
in setting ({\sf G}), and $$W_k = \{ \varphi \in L^2_1(X) : \Delta \varphi + \lambda_k \varphi = 0 \},$$
in setting ({\sf M}).
The {\em multiplicity of $\lambda_k$} is defined as $m_k = \dim(W_k)$.
Observe the difference in sign conventions, which will not disturb
us since we interact with $\Delta$ through the following two facts.

\medskip

First, if $\lambda$ is an eigenvalue of $\Delta$
with corresponding eigenfunction $\f : X \to \mathbb R$, then
\begin{equation}\label{eq:stokes}
\int |\grad \f|^2\,d\mu = \lambda \int \f^2\,d\mu.
\end{equation}

Secondly, by the min-max principle, if we have functions
$f_1, f_2, \ldots, f_{k} : X \to \mathbb R$ which have
mutually disjoint supports (and are thus linearly independent), then
we have the bound
\begin{equation}\label{eq:eigenbound}
\lambda_k \leq \max_{i=1, \ldots, k} \frac{\int |\grad f_i|^2\,d\mu}{\int (f_i-\bar{f_i})^2\,d\mu},
\end{equation}
where $\bar{f_i}= \frac{1}{\mu(X)} \int f_i\,d\mu$.
In the case $k=2$, we actually need only a single test function $f_1 : X \to \mathbb R$
in \eqref{eq:eigenbound}, since clearly $f_1 - \bar f_1$ is orthogonal to every constant function.

\medskip
\noindent
{\bf The doubling condition.}
We define $c_X = \sup \left\{ \frac{\mu(B(x,2R))}{\mu(B(x,R))} : x \in X, R > 0 \right\}$.
Without loss of generality, and for the sake of simplicity, we will assume that $c_X \geq 2$ throughout.
The next theorem follows from standard volume comparison theorems (see, e.g. \cite{Karcher89}).

\begin{theorem}\label{thm:riccidouble}
In the setting {\em({\sf M})}, if $X$ has non-negative Ricci curvature, then $c_X \leq 2^n$.
\end{theorem}

The following two facts are straightforward.

\begin{fact}\label{fact:cover}
For every $\varepsilon, R > 0$,
every ball of radius $R$ in $X$ can be covered by $c_X^{O(\log(\varepsilon^{-1}))}$ balls of radius $\varepsilon R$.
\end{fact}

\begin{fact}\label{fact:mult}
If $\mathcal B = \{B_1,  \ldots, B_M\}$ is a disjoint collection of closed balls of radius $R$, then
the intersection multiplicity of $3 \mathcal B = \{3 B_1, \ldots, 3 B_M\}$ is at most
$c_X^{O(1)}$.
\end{fact}

\smallskip
\noindent
{\bf A Poincar\'e inequality.}
Finally, we define $P_X$ as the infimum over all numbers $P$ for which the following holds:
For every $R \geq 0$, $x \in X$, and
$f : B(x,3R) \to \mathbb R$,
\begin{equation}\label{eq:pinequality}
\int_{B(x,R)} |f-\bar f_R|^2\,d\mu \leq P R^2 \int_{B(x,3R)}
|\grad f|^2 \,d\mu,
\end{equation}
where
$\bar f_R = \frac{1}{\mu(B(x,R))} \int_{B(x,R)} f\,d\mu$.

\medskip

We recall the following two known results about the relationship between $P_X$ and $c_X$.

\begin{theorem}[Kleiner and Saloff-Coste \cite{Kleiner07}]
\label{thm:poingroup}
In the setting {\em({\sf G})}, if $X$ is additionally a Cayley graph, then $P_X \lesssim c_X^3$.
\end{theorem}

\begin{theorem}[Buser \cite{Buser82}]
\label{thm:poinricci}
In the setting {\em({\sf M})}, if $X$ has non-negative Ricci curvature, then $P_X \lesssim c_X$.
\end{theorem}

\section{Eigenvalue multiplicity on doubling spaces}
\label{sec:eigen}

In this section, we prove the following.
\begin{theorem}\label{thm:main}
In both settings {\em ({\sf G})} and {\em ({\sf M})}, the multiplicity $m_k$ of the $k$th eigenvalue
of the Laplacian on $X$ satisfies
\begin{eqnarray*}
m_2 &\leq& c_X^{O(\log P_X + \log c_X)}, \quad\ \textrm{and} \\
m_k &\leq& c_X^{O(\log P_X + c_X \log k)} \quad\textrm{ for $k \geq 3$.}
\end{eqnarray*}

If in the setting {\em ({\sf G})}, $X$ is a Cayley graph, then for $k \geq 2$,
\begin{equation}\label{eq:groupmult}
m_k \leq \exp\left(O(\log c_X) (\log c_X + \log k)\right)
\end{equation}

If in the setting {\em ({\sf M})}, $X$ additionally has non-negative Ricci curvature,
then for $k \geq 2$,
\begin{equation}\label{eq:riccimult}
m_k \leq \exp\left(O(n^2 + n\log k)\right)
\end{equation}
\end{theorem}

We will require the following eigenvalue bounds.

\begin{theorem}[Cheng \cite{Cheng75}]
\label{thm:eigen-ricci}
In setting {\em ({\sf M})}, if $X$ also has non-negative Ricci curvature, then
the $k$th
eigenvalue of the Laplacian on $X$ satifies
$$
\lambda_k \lesssim \frac{k^2 n^2}{\diam(X)^2}.
$$
\end{theorem}

Cheng's result is proved via comparison to a model space of constant sectional
curvature.  In general, we can prove a weaker
bound under just a doubling assumption.  The proof is deferred to Section
\ref{sec:eigenbounds}.

\begin{theorem}
\label{thm:eigen-doubling}
In both settings {\em ({\sf G})} and {\em ({\sf M})}, the following is true.
The $k$th eigenvalue of the Laplacian on $X$ satisfies
\begin{eqnarray*}
\lambda_2 &\leq& \frac{c_X^{O(1)}}{\diam(X)^2}, \quad\textrm{and} \\
\lambda_k &\leq& \frac{k^{O(c_X)}}{\diam(X)^2} \quad\textrm{ for $k \geq 3$.}
\end{eqnarray*}
If, in addition, for every $x,y \in X$ and $R \geq 0$, we have $\mu(B(x,R)) = \mu(B(y,R))$, then
one obtains the estimate
\begin{equation}\label{eq:sym}
\lambda_k \lesssim \frac{k^2 (\log c_X)^2}{\diam(X)^2},
\end{equation}
for all $k \geq 2$.
\end{theorem}

We proceed to the proof of the theorem.

\remove{
\bigskip

[Instead, let's break $V$ into $O(k)$ partitions
$P_1, P_2, \ldots, P_m$ so that each partition
is $O(\delta D)$-bounded, but also so that every $3B_j$ ball is contained in some piece.

Then we estimate
$$
\sum_{x \in B_j} |f(x)-\Phi_j(f)|^2 \leq O(d\delta D)^2 C \sum_{x \in 3B_j} \sum_{y \sim x} |f(x)-f(y)|^2
\leq O(d\delta D)^2 \sum_{x \in \mathrm{piece}(3B_j)} \sum_{y \sim x} |f(x)-f(y)|^2.
$$
}

\begin{proof}[Proof of Theorem \ref{thm:main}]
Let $D = \diam(X)$, and let $\mathcal B = \{B_1, B_2, \ldots, B_M\}$ be a cover of $X$
of minimal size by balls of radius $\delta D$, for some $\delta > 0$ to be chosen later.
By the doubling property (and Fact \ref{fact:cover}), we have $M \leq c_X^{O(\log(\delta^{-1}))}.$

Let $W_k$ be the $k$th eigenspace of the Laplacian, and define
the linear map $\Phi : W_k \to \mathbb R^{M}$ by $\Phi_j(\f) = \frac{1}{\mu(B_j)} \int_{B_j} \f\,d\mu$.
Our goal will be to show that for $\delta > 0$ small enough, $\Phi$ is injective,
and thus $\dim(W_k) \leq M$.

\begin{lemma}\label{lem:main}
If $\f : X \to \mathbb R$ is a non-zero eigenfunction of the Laplacian with eigenvalue $\lambda\neq 0$,
and $\Phi(\f) = 0$, then
$$
\lambda^{-1} \lesssim c_X^{O(1)} P_X (\delta D)^{2}.
$$
\end{lemma}

\begin{proof}
Using $\Phi_j(\f) = 0$ for every $j \in [M]$, and the Poincar\'e inequality
\eqref{eq:pinequality}, we write
\[
\int \f^2\,d\mu \leq \sum_{j=1}^M \int_{B_j} \f^2\,d\mu
\lesssim P_X (\delta D)^2 \sum_{j=1}^M \int_{3B_j} |\grad \f|^2\,d\mu.
\]
Also,
\[
\sum_{j=1}^M \int_{3B_j} |\grad \f|^2\,d\mu \leq \mathcal M(3\mathcal B) \int |\grad \f|^2\,d\mu,
\]
where $\mathcal M(3\mathcal B) = \max_{x \in V} \# \{ j \in [M] : x \in 3B_j \} \leq c_X^{O(1)}$ is the intersection
multiplicity of $3 \mathcal B$ (by Fact \ref{fact:mult}).  Combining these two inequalities and using \eqref{eq:stokes} yields
\[
\int \f^2\,d\mu \lesssim c_X^{O(1)} P_X (\delta D)^2 \int |\grad \f|^2\,d\mu
\lesssim c_X^{O(1)} P_X (\delta D)^2 \lambda \int \f^2\,d\mu
\]
which gives the desired conclusion.
\end{proof}

Now suppose that $\f \in W_k$ and $\Phi(\f) = 0$.
If $\f \neq 0$, then by Lemma \ref{lem:main}, we have
$$
\lambda_k \gtrsim \frac{c_X^{-O(1)} P_X^{-1}}{\delta^2\, \diam(X)^2}.
$$
Choosing $\delta > 0$ small enough contradicts Theorem \ref{thm:eigen-ricci} or
Theorem \ref{thm:eigen-doubling}, depending upon the assumption.
It follows that $\dim(W_k) \leq M \leq c_X^{O(\log(\delta^{-1}))}$, yielding
the desired bounds.

To prove \eqref{eq:riccimult}, use Theorem \ref{thm:eigen-ricci}, and observe that $P_X \lesssim c_X$
by Theorem \ref{thm:poinricci}.
To obtain \eqref{eq:groupmult}, observe that $P_X \lesssim c_X^3$, by Theorem \ref{thm:poingroup},
and the condition of the eigenvalue estimate \eqref{eq:sym} is satisfied when $X$ is
a Cayley graph (indeed, for any vertex-transitive graph).
\end{proof}

\remove{
\section{Multiplicity of eigenvalues on manifolds of Ricci curvature bounded below}

We seek to prove the following theorem.

\begin{theorem}
Let $M$ be an $n$-dimensional compact Riemannian manifold without boundary, let $d = \diam(M)$, and suppose
that $(\mathrm{Ricc})d^2 \geq - a^2$, where $\mathrm{Ricc}$ is a lower bound for the
Ricci curvature on $M$.  Then,
$$
m_k(M) \leq C(a,n,k),
$$
where $m_k$ is the multiplicity of the $k$th eigenvalue, and
$C(a,n,k)$ is a constant depending only on $a$, $n$, and $k$.
\end{theorem}

First, we need a refined version of the Reverse Poincar\'e inequality.

\begin{lemma}
Let $u$ be an eigenfunction of the Laplace operator with
eigenvalue $\lambda$.  Consider a ball $B(R)$, then
$$
\int_{B(R)} |\grad u|^2\, dV \leq \left(\frac{16}{R^2} + 2 \lambda\right) \int_{B(2R)} u^2\,dV.
$$
\end{lemma}

We also a regular Poincar\'e inequality.

\begin{lemma}
For any $f \in W_{\mathrm{loc}}^{1,2}(M)$ [see CM-Annals],
$$
\int_{B(R)} (f-\bar f)^2\,dV \leq C r^2 \int_{B(R)} |\grad f|^2\, dV
$$
\end{lemma}
}

\subsection{Aside:  A Reverse Poincar\'e Inequality for graphs}
\label{sec:reverse}

In the approaches of Colding and Minicozzi \cite{CM97} and Kleiner \cite{Kleiner07},
one also needs a ``reverse Poincar\'e inequality'' to control harmonic function on
balls, while in the preceding proof we only need control of an eigenfunction
on the entire graph (for which we could use \eqref{eq:stokes}).
We observe the following (perhaps known) version for eigenfunctions on graphs.
An analogous statement holds in setting {\sf (M)}.

\begin{theorem}\label{thm:reverse}
Suppose we are in the graph setting {\sf (G)}.
Let $\f : X \to \mathbb R$ be an eigenfunction of the Laplace operator with eigenvalue
$\lambda$. Then,
$$\int_{B(R)} |\grad \f|^2 \,d\mu \leq \left(\frac{128}{d R^2} + 2 \lambda\right) \int_{B(2R)} \f^2\,d\mu.$$
\end{theorem}

The proof is based on the following lemma.

\begin{lemma}
\label{lem:bound}
Let $\f : X \to \mathbb R$ be an eigenfunction of the Laplace operator with eigenvalue
$\lambda$. Let $u : X \to \mathbb R$ be a non-negative function that vanishes off $B(R-1)$, then
$$\int_{B(R)} u^2 |\grad \f|^2\,d\mu \leq  \frac{128}{d} \int_{B(R)} \f^2 |\grad u|^2 \,d\mu+ 2\lambda  \|u\|_{\infty}^2 \int_{B(R)} \f^2\,d\mu.$$
\end{lemma}
\begin{proof}
Denote $S = 2d \int_{B(R)} u |\grad \f|^2\,d\mu$. Assume $u$ vanishes off
$B(R-1)$. We have,
\begin{align*}
S &= \int_{B(R)} \,\sum_{y \sim x} u(x)^2 |\f(x) - \f(y)|^2 \,d\mu(x)  \\
&= \int_{B(R)} \,\sum_{y \sim x} u(x)^2 (\f(x)^2 + \f(y)^2 - 2 \f(x) \f(y)) \,d\mu(x)  \\
&=
\int_{B(R)} \sum_{y \sim x} u(x)^2 \f(x)^2 \,d\mu(x)
+ \int_{B(R)} \sum_{y \sim x} u(x)^2 \f(y)^2 \,d\mu(x)
-  2\int_{B(R)} \sum_{y \sim x} u(x)^2 \f(x) \f(y)\,d\mu(x) \\
&=
\int_{B(R)} \sum_{y \sim x} u(x)^2 \f(x)^2 \,d\mu(x)
+ \int_{B(R)} \sum_{y \sim x} u(y)^2 \f(x)^2 \,d\mu(x)
-  2\int_{B(R)} \sum_{y \sim x} u(x)^2 \f(x) \f(y)\,d\mu(x) \\
&= 2 \int_{B(R)} \sum_{y \sim x} u(x)^2 \f(x) (\f(x)-\f(y)) \,d\mu(x)+
\int_{B(R)} \sum_{y \sim x} (u(y)^2-u(x)^2) \f(x)^2\,d\mu(x).
\end{align*}
Here we used that $u(x)^2 \f(x)^2 = u(y)^2 \f(x)^2 = 0$ when $y \notin B(R)$,
since $u$ vanishes off $B(R-1)$. First, let us bound the first term.
\begin{eqnarray*}
\int_{B(R)} \sum_{y\sim x} u(x)^2 \f(x) (\f(x) - \f(y)) \,d\mu(x)&=&
\int_{B(R)} u(x)^2 \f(x) \left(\sum_{y\sim x} \f(x) - \f(y)\right)d\mu(x)\\
&=& d \int_{B(R)} u^2 \f \laplace \f\,d\mu \\
&=& d \int_{B(R)} u^2 \lambda \f^2 \,d\mu\\
&\leq& d\lambda \|u\|_{\infty}^2
\int_{B(R)} \f^2\,d\mu.
\end{eqnarray*}

Now we bound the second term.
\begin{align*}\int_{B(R)} & \sum_{y\sim x} (u(y)^2 - u(x)^2) \f(x)^2 \,d\mu(x) \\
&\leq 2  \int_{B(R)} \sum_{y\sim x} (u(y)^2 - u(x)^2) (\f(x)^2 - \f(y)^2) \\
&= 2 \left(\int_{B(R)} \sum_{y\sim x} |u(x) + u(y)|^2 |\f(x)  - \f(y)|^2\,d\mu(x)\right)^{1/2} \\
&\qquad\times\left(\int_{B(R)} \sum_{y\sim x} |u(x) - u(y)|^2 |\f(x) + \f(y)|^2\,d\mu(x)\right)^{1/2} \\
&\leq
8
\left(\int_{B(R)} \sum_{y\sim x} u(x)^2 |\f(x)  - \f(y)|^2\,d\mu(x)\right)^{1/2} 
\left(\int_{B(R)} \sum_{y\sim x} |u(x) - u(y)|^2 |\f(x)|^2\,d\mu(x)\right)^{1/2} \\
& = 8 S^{1/2} \left(\int_{B(R)} |\grad u|^2 \f^2\,d\mu\right)^{1/2}.
\end{align*}

Combining these bounds we get,
$$S \leq 2 d \lambda \|u\|_{\infty}^2 \int_{B(R)} \f^2\,d\mu+  8S^{1/2}
\left(\int_{B(R)} |\grad u|^2 \f^2\,d\mu\right)^{1/2}.$$
Therefore, either
$$S \leq 4 d \lambda \|u\|_{\infty}^2 \int_{B(R)} \f^2\,d\mu$$
and then we are done, or
$$S \leq 16 S^{1/2} \left(\int_{B(R)} |\grad u|^2 \f^2\,d\mu\right)^{1/2}.$$
Then $S \leq 256 \int_{B(R)} |\grad u|^2 \f^2\,d\mu$.
\end{proof}

\begin{proof}[Proof of Theorem \ref{thm:reverse}]
The theorem follows from Lemma~\ref{lem:bound}, if we choose
$$u(x) =
\begin{cases}
1,& \text{if } x\in B(R)\\
1 - d(x,B(R))/R, &\text{if } x\in B(2R) \setminus B(R).
\end{cases}
$$
\end{proof}

\subsection{Eigenvalue bounds}
\label{sec:eigenbounds}

We now proceed to prove the eigenvalue bounds of Theorem \ref{thm:eigen-doubling}.
The following lemma is essentially proved in \cite{GKL03}; a similar
statement with worse quantitative dependence can be
deduced from \cite{Ass83}.

\begin{lemma}\label{lem:pad}
There exists a constant $A \geq 1$ such that the following holds.
Let $(X,\dist,\mu)$ be any compact metric-measure space, where $\mu$
satisfies the doubling condition with constant $c_X$.
Then for any $\tau > 0$, there exists a finite partition $P$ of $X$
into $\mu$-measurable subsets such that the following holds.
If $S \in P$, then $\diam(S) \leq \tau$.  Furthermore,
if we use $P(x)$ denote the set $P(x) \in P$ which contains $x \in X$, then
\begin{equation}\label{eq:goodpart}
\mu \left(\left\{ x \in X : \dist(x, X \setminus P(x)) \geq \frac{\tau}{A(1+\log(c_X))}\right\}\right) \geq \frac12.
\end{equation}
\end{lemma}

\begin{proof}
We will first define a {\em random} partition of $X$ as follows.
Let $N = \{x_1, x_2, \ldots, x_M\}$ be a $\tau/4$-net in $X$, and choose a {\em uniformly random}
bijection $\pi : [M] \to [M]$. Also let $\alpha \in [\frac14, \frac12]$
be chosen uniformly at random, and inductively define
$$
S_i = B(x_{\pi(i)}, \alpha \tau) \setminus \bigcup_{j=1}^{i-1} S_j.
$$
It is clear that $P = S_1 \cup S_2 \cup \cdots \cup S_M$ forms a partition of $X$
(note that some of the sets may be empty),
and $\diam(S_i) \leq \tau$ for each $i$.
Note that the distribution of $P$ is independent of the measure $\mu$.

\begin{claim}\label{claim:CKR}
For some $A \geq 1$, and every $x \in X$,
\begin{equation}\label{eq:prob}
\Pr_P\left[ \dist(x, X \setminus P(x)) \geq \frac{\tau}{A(1+\log(c_X))} \right] \geq \frac12.
\end{equation}
\end{claim}

By averaging, the claim implies that \eqref{eq:goodpart} holds for some partition $P$
of the required form.  For the sake of completeness, we include here a simple proof
of Claim \ref{claim:CKR}, which essentially follows from \cite{CKR01}.

\begin{proof}[Proof of Claim \ref{claim:CKR}]
Fix a point $x \in X$ and some value $t \leq \tau/8$. 
Observe that, by Fact \ref{fact:cover}, we have $m = |N \cap B(x,\tau)| \leq c_X^{O(1)}$.
 Order the points
of $N \cap B(x,\tau)$ in increasing distance from $x$: $w_1, w_2, \ldots, w_m$.
Let $I_k = [d(x,w_k)-t,d(x,w_k)+t]$ and write $\mathcal E_k$ for the event
that $\alpha \tau \leq d(x,w_k) + t$ and $w_k$ is the minimal element
according to $\pi$ for which $\alpha \tau \geq d(x,w_k) -t$.
It is straightforward to check that the event $\left\{d(x,X\setminus P(x)) \leq t\right\}$
is contained in the event $\bigcup_{k=1}^m \mathcal E_k$.
Therefore,
\begin{eqnarray}
\Pr\left[d(x, X \setminus P(x)) \leq t\right] \leq \sum_{k=1}^m \mathcal \Pr[\mathcal E_k]
&=& \sum_{k=1}^m \Pr[\alpha \tau \in I_k] \cdot \Pr[\mathcal E_k\,|\,\alpha\tau \in I_k] \nonumber \\
&\leq &
\sum_{k=1}^m \frac{2t}{\tau/4} \frac{1}{k} \leq \frac{8t}{\tau} (1+\log m), \label{eq:finalCKR}
\end{eqnarray}
where we have used the fact that $$\Pr[\mathcal E_k\,|\, \alpha \tau \in I_k] \leq \Pr[\min\{\pi(i) : i=1,2,\ldots,k\} = \pi(k)] = 1/k.$$
Thus choosing $t \approx \frac{\tau}{1+\log(c_X)}$ in \eqref{eq:finalCKR} yields the desired bound \eqref{eq:prob}.
\end{proof}
\end{proof}

\remove{
\begin{proof}
[Need to change this {\em slightly,} but it still works...]
Let $N = \{x_1, x_2, \ldots, x_M\}$ be a $\tau/4$-net in $X$.
Let $k = \lfloor 1 + \log(c_X)\rfloor$, and for $i=0,1,\ldots, k$,
let
$$
B_i = B\left(x_1, \frac{\tau}{4} + \frac{i\tau}{4(1+ \log(c_X))}\right).
$$
Clearly $B(x_1,\tau/4) = B_0 \subseteq B_1 \subseteq \cdots \subseteq B_k \subseteq B(x_1,\tau/2)$.
There must exist some value $0 \leq i \leq k-1$ for which $\mu(B_{i+1}) \leq 2 \mu(B_{i+1})$.
Otherwise, $$\mu(B(x_1, \tau/2)) \geq \mu(B_k) > 2^{\log(c_X)} \mu(B_0) = c_X \mu(B(x_1,\tau/4)),$$
violating the doubling assumption.  Let $S_1 = B_{i+1}$, and observe that
$$
\mu\left(\left\{ x \in S_1 : \dist(x, X\setminus S_1) \geq \frac{\tau}{4(1+\log(c_x))}\right\}\right) \geq \mu(B_i) \geq \frac12 \mu(S_1).
$$
By inducting on $X \setminus S_1$ with the $\tau/4$-net $N \setminus \{x_1\}$, we obtain
a partition satisfying \eqref{eq:goodpart}.
\end{proof}
}

The next simple lemma shows that on a ``coarsely path-connected'' space,
a doubling measure cannot be concentrated on very small balls.

\begin{lemma}\label{lem:measlem}
Let $(X,\dist,\mu)$ satisfy {\em ({\sf G})} or {\em ({\sf M})}.
Then for any $x \in X$ and $10 \leq R \leq D$, we have
$$
\left(1-\frac{1}{2c_X}\right)\mu(B(x, R)) \geq \mu(B(x,R/10)).
$$
\end{lemma}

\begin{proof}
Let $\delta = 1/(2c_X)$.
Suppose there is an $x \in X$ with $\mu(B(x,R/10)) \geq (1-\delta) \mu(B(x,R))$, and
$10 \leq R \leq D$.  We may assume that $\mu(B(x,R))=1$.
In both settings {({\sf G})} and {({\sf M})}, there exists a $y \in X$ such that $3R/5 \geq \dist(x,y) \geq R/2$.
Let $r = 3R/8$ so that $B(y,2r) \supseteq B(x,R/10)$ but $B(y,r) \subseteq B(x,R) \setminus B(x,R/10)$.

In this case, $\mu(B(y, r)) \leq \mu(B(x,R)) - \mu(B(x,R/10)) \leq \delta$, and
$$\mu(B(y, 2r)) = \mu(B(y,3R/4)) > \mu(B(x, R/10)) \geq 1-\delta \geq \frac{1-\delta}{\delta} \mu(B(y,r)).$$
Since $(1-\delta)/\delta \geq c_X$, this violates the doubling assumption,
yielding a contradiction.
\end{proof}

\begin{cor}\label{cor:smallmeas}
Let $(X,\dist,\mu)$ satisfy {\em ({\sf G})} or {\em ({\sf M})}.
Then for any $x \in X$ and any $\varepsilon > 0$, we have
$$
\mu(B(x, \varepsilon\,\diam(X))) \leq 1+ \mu(X) \left(1-\frac{1}{2c_X}\right)^{O(\log(\varepsilon^{-1}))}.
$$
\end{cor}

Under a symmetry assumption, there is an obvious improvement.

\begin{lemma}\label{lem:meassym}
Let $(X,\dist,\mu)$ satisfy {\em ({\sf G})} or {\em ({\sf M})}.
If, for every $x,y \in X$ and $R \geq 0$, we have $\mu(B(x,R)) = \mu(B(y,R))$,
then for every $\varepsilon > 0$,
$$
\mu(B(x, \varepsilon\,\diam(X))) \lesssim 1+\epsilon\,\mu(X).
$$
\end{lemma}

\begin{proof}
Fix $x$ and $y$ with $\dist(x,y) = \diam(X)$,
and connect $x$ and $y$ by a geodesic $\gamma$.
Let $N \subseteq \gamma$ be a maximal $(3\varepsilon\, \diam(X))$-separated set,
so that $|N| \gtrsim 1/\varepsilon$.
Then the balls $\{B(u,\varepsilon\, \diam(X))\}_{u \in N}$ are disjoint, and each of equal measure, implying the claim.
\end{proof}

We now prove Theorem \ref{thm:eigen-doubling}, yielding upper bounds on the eigenvalues of $\Delta$.

\begin{proof}[Proof of Theorem \ref{thm:eigen-doubling}]
Use Corollary \ref{cor:smallmeas} to choose
\begin{equation}\label{eq:choosetau}
\tau \geq \frac{\diam(X)}{e^{O(c_X \log(k))}}
\end{equation}
so that for every $x \in X$,
\begin{equation}\label{eq:size}
\mu(B(x,2\tau)) \leq \frac{\mu(X)}{8k}.
\end{equation}
Let $P$ be the partition guaranteed by Lemma \ref{lem:pad} with parameter $\tau$.
Since every $S \in P$ satisfies $\diam(S) \leq \tau$, \eqref{eq:size} implies
that $\mu(S) \leq \mu(X)/(8k)$.
Call a set $S \subseteq X$ {\em good} if it satisfies
\begin{equation}\label{eq:good}
\mu \left(\left\{ x \in S : \dist(x, X \setminus S) \geq \frac{\tau}{A(1+\log(c_X))}\right\}\right) \geq \frac14 \mu(S).
\end{equation}
By averaging, at least $1/4$ of the measure is concentrated on good sets $S \in P$.

In particular, since every $S \in P$ satisfies $\mu(S) \leq \mu(X)/(8k)$, from
the good sets $S \in P$, we can form (by taking unions of small sets) disjoint sets $S_1, S_2, \ldots, S_k$
such that each $S_i$ is good and satisfies
\begin{equation}
\label{eq:size2}
\frac{\mu(X)}{8k} \leq \mu(S_i) \leq \frac{\mu(X)}{4k}.
\end{equation}

Now define $f_i : X \to \mathbb R$ by $f_i(x) = \dist(x, X \setminus S_i)$.  Clearly the $f_i$'s have disjoint support.
Furthermore, each $f_i$ is 1-Lipschitz, hence
$
\int |\grad f_i|^2\,d\mu \leq \mu(X).
$
Finally, since each set $S_i$ is good and satisfies \eqref{eq:size2}, we have
$$
\int (f_i - \bar f_i)^2\,d\mu \gtrsim \frac{\tau^2}{(\log c_X)^2} \frac{\mu(X)}{k}.
$$
Using \eqref{eq:eigenbound}, this implies that
$$
\lambda_k \leq \frac{k^{O(c_X)}}{\diam(X)^2}.
$$

Observe that we can obtain a better bound
$$
\lambda_2 \leq \frac{c_X^{O(1)}}{\diam(X)^2}
$$
as follows.
In this case, we only need one test function.
Choose $\tau = \diam(X)/20$ above, and use Lemma \ref{lem:measlem} to
form a good set $S_1$ which satisfies
$$
\frac{\mu(X)}{2c_X} \leq \mu(S_1) \leq \left(1-\frac{1}{2c_X}\right) \mu(X),
$$
then define $f_1(x) = \dist(x, X \setminus S_1)$.

Finally, to prove \eqref{eq:sym}, note that under the measure symmetry assumption,
we can employ Lemma \ref{lem:meassym} to choose $\tau \geq \frac{\diam(X)}{O(k)}$
in \eqref{eq:choosetau}.  The rest of the proof proceeds exactly as before.
\end{proof}

\remove{
\begin{theorem}\label{thm:eigen}
Let $(X,\dist,\mu)$ be a compact, geodesic metric-measure space, with a Laplace structure.
Suppose that $\mu$ is doubling with constant $C \geq 1$.
If $\lambda_k$ is the $k$th eigenvalue of the Laplacian, then
$$
\lambda_k \leq \frac{k^{O(C)}}{\diam(X)^2}.
$$
\end{theorem}

\begin{proof}[Proof sketch:]
Decompose $X$ into clusters of diameter $\Delta/\mathrm{poly}(k e^C)$ so that each cluster
contains at most $n/(2k)$ points.  Then group the small clusters together to get $k$
clusters each with $\Omega(n/k)$ points.  Now construct disjoint supported bump functions
on these clusters.  Since they are disjointly supported, these test functions bound the $k$th
eigenvalue.  The test functions are good because of the padding.
\end{proof}
}


\section{Applications to finite groups}
\label{sec:groupapp}

We now give some applications of Theorem \ref{thm:main} to
finite groups.

\begin{theorem}\label{thm:grouprep}
Let $G$ be a finite group with symmetric generating set $S$.
Let $c_G = \max_{R > 0} \frac{|B(2R)|}{|B(R)|}$ be the
doubling constant of the Cayley
graph $\mathrm{Cay}(G;S)$. Then there exists
a finite-dimensional representation $\rho_W:G\to GL(W)$ such that
\begin{enumerate}
\item $\dim\, W \leq \exp(O(\log c_G)^2)$,
\item $|\rho_W(G)| \gtrsim |G|^{1/\log_2 c_G}/c_G^{O(1)}$.
\end{enumerate}
\end{theorem}
\begin{proof}
Without loss of generality, we assume that $c_G \geq 2$ throughout.
Recall that $d = |S|$.

Consider the action of $G$ on $L^2(G)$ via $[\rho(g) f](x) = f(g^{-1} x)$.
Note that this action commutes with the Laplacian,
$$[\Delta \rho(g) f](x) = [\Delta f](g^{-1} x) = f(g^{-1} x) - \frac{1}{d} \sum_{s\in S} f(g^{-1}xs)=
[\Delta f](g^{-1} x)  = [\rho(g) \Delta f] (x).$$
Therefore, every eigenspace of the Laplacian is invariant under the action of $G$.
Now let $W\equiv W_2$ and $\lambda_2$ be the second eigenspace
and eigenvalue, respectively, of the  Laplacian on $\mathrm{Cay}(G;S)$. Let $\rho_W$ be the restriction of $\rho$ to $W$.
First, by Theorem~\ref{thm:main}\eqref{eq:groupmult}, $\dim\, W \leq \exp(O(\log c_G)^2)$.

Now we need to
prove the lower bound on $|\rho_W(G)|$.
By Theorem \ref{thm:eigen-doubling}, we have
\begin{equation}
\label{eqn:ubound}
\lambda_2 \lesssim \frac{c_G^{O(1)}}{\diam(\mathrm{Cay}(G;S))^2}.
\end{equation}

Consider $H = \ker \rho_W$, the set of elements which act trivially on $W$.
$H$ is a normal subgroup of $G$ and $\rho_W(G) \cong G/H$.
Let $f$ be an arbitrary non-zero function in $W_2$.
Note that $f$ is constant on every coset $Hg$ since the value of $f(hg) = [\rho(h^{-1}) f] (g) = f(g)$ does not depend on $h\in H$. Define $\hat f:G/H\to {\mathbb R}$ by ${\hat f}(Hg) = f(g)$.
Observe that $\hat f$ is
a non-constant eigenfunction of the Laplacian on the quotient graph $\mathrm{Cay}(G/H;S)$ with eigenvalue $\lambda_2$,
$$\Delta \hat f (Hg) = \hat f(Hg) - \frac{1}{d} \sum_{s\in S} \hat f(Hgs) =
f(g) - \frac{1}{d} \sum_{s\in S} f(gs) = \Delta f (g) = \lambda_2 f(g) = \lambda_2 \hat f(Hg).$$

Let $\lambda_2(G/H)$ denote the second eigenvalue of the Laplacian on $\mathrm{Cay}(G/H;S)$.
Since $\lambda_2$ is a non-zero eigenvalue of the Laplacian on $\mathrm{Cay}(G/H;S)$,
we have $\lambda_2(G/H) \leq \lambda_2$.
However, by the discrete Cheeger inequality \cite{AM85},
\begin{equation}\label{eq:cheeger}
\lambda_2(G/H) \geq \frac{h(\mathrm{Cay}(G/H;S))^2}{2d^2}
\end{equation}
where
$h(\mathrm{Cay}(G/H;S))$ is the Cheeger constant of $\mathrm{Cay}(G/H;S)$:
$$h(\mathrm{Cay}(G/H;S)) \equiv \max_{U\subset G/H; |U| \leq |G/H|/2} \frac{E(U, (G/H)\setminus U)}{|U|} \geq \frac{1}{|G/H|/2},$$
here $E(U, (G/H)\setminus U)$ denotes the set of edges between $U$ and $(G/H)\setminus U$ in
$\mathrm{Cay}(G/H;S)$,
and the bound follows because $\mathrm{Cay}(G/H;S)$ is a connected graph.

We conclude that
$$\lambda_2 \geq \lambda_2(G/H) \geq \frac{(2/|G/H|)^2}{2d^2} = \frac{2}{(d|G/H|)^2}.$$
Combining this bound with (\ref{eqn:ubound}), we get
$$|\rho_W(G)| = |G/H| \geq \sqrt{\frac{2}{d^2 \lambda_2}} \gtrsim \frac{\diam(\mathrm{Cay}(G;S))}{c_G^{O(1)}}.$$
The desired bound now follows using the fact that
$\diam (\mathrm{Cay}(G;S)) \geq |G|^{1/\log_2 c_G}$.
\end{proof}

\begin{cor}\label{cor:Zhom}
Under the assumptions of Theorem \ref{thm:grouprep}, there exists a normal subgroup
$N$ with $[G : N] \leq \alpha$ such that $N$ has $\mathbb Z_M$ as a homomorphic image,
where $M \gtrsim |G|^{\delta}$ and $\delta = \delta(c_G)$ and $\alpha = \alpha(c_G)$
depend only on the doubling constant of $G$.
\end{cor}

\begin{proof}
Let $\rho_W : G \to GL(W)$ be the representation guaranteed by Theorem \ref{thm:grouprep},
and put $k = \dim\, W$.  Now, $H=\rho_W(G)$ is a finite subgroup of $GL(W)$, hence by a theorem of Jordan (see \cite[36.13]{CR06}),
$H$ contains a normal abelian subgroup $A$ with $[H : A] = O(k)^{k^2}$.
Since $A$ is abelian, its members can be simultaneously diagonalized over $\mathbb C$; it follows that $A$
is a product of at most $k$ cyclic groups, hence $\mathbb Z_M \leq A$ for some $M \geq |A|^{1/k}$.
Putting $N = \rho^{-1}(A)$, we see that
$[G : N] = [H : A] = O(k)^{k^2}$, and $N$ maps homomorphically onto $\mathbb Z_M$.
\end{proof}

\subsection*{Acknowledgements}

We are grateful to Luca Trevisan for bringing the questions about
a finitary analog of Gromov's theorem to our attention.  We also thank
Bruce Kleiner for helpful bibliographical remarks.

\bibliographystyle{abbrv}
\bibliography{gromov}

\end{document}